\documentclass[12pt,a4paper]{article}
\usepackage{authblk}
\usepackage{amsmath}
\usepackage{amsfonts}
\usepackage{amssymb,float}
\usepackage{latexsym}
\usepackage[latin1]{inputenc}
\usepackage{amssymb,float}
\usepackage{latexsym}
\usepackage{epstopdf}
\usepackage{geometry}
\usepackage[active]{srcltx}
\usepackage{graphicx}
\usepackage{epstopdf}
\usepackage{wrapfig}
\usepackage{caption}
\usepackage[T1]{fontenc}
\usepackage{subcaption}
\usepackage{enumerate}
\usepackage[ bookmarks=true,         bookmarksnumbered=true, colorlinks=true,
pdfstartview=FitV, linkcolor=blue, citecolor=blue, urlcolor=blue]{hyperref}
\usepackage{amssymb}%
\providecommand{\U}[1]{\protect\rule{.1in}{.1in}}
\newtheorem{theorem}{Theorem}[section]
\newtheorem{proposition}{Proposition}[section]

\newtheorem{corollary}{Corollary}[section]

\newenvironment{proof}[0]{\paragraph{Proof.}}{\rule{0.5em}{0.5em}}

\begin{document}
\title{\bf Property Of The Beta Modified Weibull Distribution With Six Parameters}
\author[1]{Didier Alain Njamen Njomen\thanks{didiernjamen1@gmail.com}}
\author[1,2]{Fidel Djongreba Ndikwa\thanks{fidel.djong@yahoo.fr}}
\affil[1]{Department of Mathematics and Computer Science, University of Maroua, Maroua, Cameroon}
\affil[2]{Department of Operational Research, Higher Institute of Transport Logistics and Commerce, University of Ebolowa, Ebolowa, Cameroon}

\maketitle

\begin{abstract}
The aim of this article is to determine a new six-parameter
Beta Weibull distribution and its various associated functions, 
namely the cumulative distribution, survival, probability density 
and hazard functions. Next, we determine the
sub-distributions of the new distribution and show that the
latter generalizes those of the literature. Finally, numerical
simulations were performed and show that the shapes of the
density function of the new distribution cover all those in the literature, and the shapes of hazard functions (constant, increasing, decreasing, $\bigcup$-shaped and $\bigcap$-shaped) are represented in the new distribution and encompass all existing distributions.
\end{abstract}



\section{Introduction}
\label{intro}

Non-parametric estimation of the density and risk function has been
extensively studied in the literature in several forms: circular kernel \cite{Njomen}, noise distribution in CHARN models \cite{Ngatchou}. This estimation was used in the Weibull distribution \cite{Weibull}, which is one of the main families used to model phenomena in biology and reliability engineering. However, one of the limitations of this distribution is that the hazard functions generated from it are either increasing or strictly decreasing monotonic. To obtain different forms of hazard functions, it is possible to extend the Weibull distribution by adding parameters, which allows for more flexible forms because, as \cite{Khan} points out, adding a parameter to an existing distribution makes it more flexible. Thus, the Generalized Beta approach was introduced to solve this type of problem. We cite, for example, the Beta Weibull distributions, Beta Exponentiel, Beta Gumbel, Beta
Normal...(see \cite{Eugene}, \cite{Famoye}, \cite{Gupta1}, \cite{Khan}, \cite{Lai}, \cite{Lee}, \cite{Mudholkar1}, \cite{Mudholkar2}, \cite{Mudholkar}, \cite{Nadarajah2}, \cite{Nadarajah3}, \cite{Nadarajah1} and \cite{Djongreba}).

A review of the literature shows that several forms of the modified Weibull Beta distribution exist, such as those given by \cite{Famoye}, \cite{Wahed}, \cite{Silva}, \cite{Nadarajah2}, \cite{Khan} and \cite{Lai}. However, one aspect of the Weibull distribution (see \cite{Jeong}) has not yet been used. This article is written in this context.

The aim of this article is to use the \cite{Jeong} distribution to obtain a new six-parameter Beta Weibull distribution, generalizing all existing Beta Weibull distributions.

The method used in this article is from \cite{Silva} as are the references cited therein.

The rest of this article is structured as follows: In Section 2, we estimate the modified Beta Weibull distribution and in Section 3, we give the main results relating to the new six-parameter distribution. Next in Section 4, numerical simulations of the density and hazard functions of the new distribution are carried out in order to compare it with those in the literature. Finally, in Section 5, we give some special cases which illustrate the sub-distributions of the new distribution and which generalize all those in the literature.

\section{Estimation of the beta modified Weibull distribution}
\label{definition}

It is clearly demonstrated in \cite{Wahed} that, as $\tau$ approaches zero, the
four-parameter Weibull distribution reduces to the generalized
Weibull distribution of [12] and covers a greater number of shapes
than the three-parameter extension of the Weibull model and is
regular, regardless of the parameter values. This proves the flexibility of the four-parameter Weibull model, allowing us to study its
extensions, hence the interest in studying new models.

Let $X$ be a non-negative random variable, $\Omega$ the parameter
space, and $x$ the realization of $X$. Let $G(x)$ be the cumulative distribution
function of the variable $X$. We assume that X follows the modified
Weibull distribution with parameters $(\gamma,\beta,\lambda,\tau)
\in \Omega$ defined in \cite{Jeong} and $G(X)$ follows the Beta distribution with parameters $(a,b) \in \Omega$.
The survival and hazard functions of the four-parameter Weibull
distribution defined in \cite{Jeong} are given respectively by the following
relations \eqref{dnf0} et \eqref{dnf1}:

\begin{eqnarray}
\label{dnf0} S(x;\gamma,\beta,\lambda,\tau)=
\exp\left[-\dfrac{\lambda^{1-\tau} \left( (\frac{x}{\beta})^{\gamma}
+ \lambda \right)^{\tau} - \lambda }{\tau} \right],
\end{eqnarray}
with
$$\quad\quad 0<x<\infty , 0<\gamma,\beta,\lambda<\infty , -\infty<\tau<\infty ,$$
and
\begin{eqnarray}
\label{dnf1} h(x;\gamma,\beta,\lambda,\tau)=
\dfrac{\gamma(\frac{x}{\beta})^{\gamma} \lambda^{1-\tau} }
{x\left[(\frac{x}{\beta})^{\gamma} + \lambda \right]^{1-\tau}}
.\end{eqnarray}

In this article, we consider $G(x)$ to be the cumulative distribution
function of the modified Weibull distribution and the Beta
distribution with two parameters $a$ and $b$ such that:
$$B_{1}(a,b)=\int_{0}^{1}x^{a-1}(1-x)^{b-1}dx, \quad a,b>0.$$
A variable $X(0\leq x \leq \infty)$ is a variable $\beta(a,b)$ if its elementary
cumulative probability distribution is given by :
$$f(x,a,b)=\frac{x^{a-1}(1-x)^{b-1}}{B(a,b)},$$
and its cumulative distribution function is defined by :

\begin{eqnarray*}
F(x,a,b) = \frac{1}{B(a,b)}\int_{0}^{x} u^{a-1}(1-u)^{b-1}du.
\end{eqnarray*}
By setting $B_{x}(a,b)=\int_{0}^{x} u^{a-1}(1-u)^{b-1}du$, we get:
\begin{eqnarray*}
F(x,a,b) = \frac{B_{x}(a,b)}{B(a,b)}.
\end{eqnarray*}
The expression $B_{x}(a,b)$ is called the incomplete beta function.

Assuming that $G(X)$ follows the Beta distribution with parameters $(a,b) \in \Omega$, then

\begin{eqnarray}
\label{dnf00}F(G(x),a,b) &=& \frac{B_{G(x)}(a,b)}{B(a,b)} \nonumber\\
&=&\frac{1}{B(a,b)}\int_{0}^{G(x)} u^{a-1}(1-u)^{b-1}du
\end{eqnarray}

\section{Main Results}

To define the new distribution, we will need the cumulative distribution
and density functions of the four-parameter Weibull distribution
from \cite{Jeong}.

\begin{proposition}
\label{prop1} 
The cumulative distribution function and probability density function associated with the four-parameter Weibull
distribution are given respectively by :
\begin{eqnarray}
\label{dnf3} G(x;\gamma,\beta,\lambda,\tau)=1-
\exp\left[-\dfrac{\lambda^{1-\tau} \left( (\frac{x}{\beta})^{\gamma}
+ \lambda \right)^{\tau} - \lambda }{\tau} \right].
\end{eqnarray}
and
\begin{eqnarray}
\label{dnf4} g(x;\gamma,\beta,\lambda,\tau)=
\dfrac{\frac{\gamma}{\beta}(\frac{x}{\beta})^{\gamma-1}
 \lambda^{1-\tau}}
{\left[(\frac{x}{\beta})^{\gamma} + \lambda \right]^{1-\tau}}
\exp\left[-\dfrac{\lambda^{1-\tau} \left( (\frac{x}{\beta})^{\gamma}
+ \lambda \right)^{\tau} - \lambda }{\tau} \right].
\end{eqnarray}

\end{proposition}

\begin{proof}
By definition, the cumulative distribution function is
given by :
$G(x;\gamma,\beta,\lambda,\tau)=1-S(x;\gamma,\beta,\lambda,\tau)$.
Therefore according to \eqref{dnf0}, we have:

\begin{eqnarray*}
G(x;\gamma,\beta,\lambda,\tau)=1-
\exp\left[-\dfrac{\lambda^{1-\tau} \left( (\frac{x}{\beta})^{\gamma}
+ \lambda \right)^{\tau} - \lambda }{\tau} \right].
\end{eqnarray*}

On the other hand, the probability density function is given by $g(x)=h(x)S(x)$. According to \eqref{dnf1}, we obtain:

\begin{eqnarray*}
g(x;\gamma,\beta,\lambda,\tau)=
\dfrac{\frac{\gamma}{\beta}(\frac{x}{\beta})^{\gamma-1}
 \lambda^{1-\tau}}
{\left[(\frac{x}{\beta})^{\gamma} + \lambda \right]^{1-\tau}}
\exp\left[-\dfrac{\lambda^{1-\tau} \left( (\frac{x}{\beta})^{\gamma}
+ \lambda \right)^{\tau} - \lambda }{\tau} \right].
\end{eqnarray*} \end{proof}

A review of the literature shows that \cite{Jeong} determined the 4-
parameter survival and risk functions; \cite{Silva} determined the 5-parameter Beta modified Weibull distribution.

In this article, we determine the six-parameter Beta modified Weibull distribution. To achieve this, we use the expression \eqref{dnf00} above, which allows us to add two new parameters to an existing distribution. Indeed, by replacing the expression \eqref{dnf3} in \eqref{dnf00}, we obtain the following result which gives the cumulative distribution and survival functions with 6 parameters respectively:

\begin{theorem}
\label{prop2} The cumulative distribution function and sur-
vival function associated with the six-parameter Beta Weibull dis-
tribution are given respectively by :
\begin{eqnarray}
\label{dnf6} F(x;a,b,\gamma,\beta,\lambda,\tau) =
\frac{1}{B(a,b)}\int_{0}^{1- \exp\left[-\dfrac{\lambda^{1-\tau}
\left( (\frac{x}{\beta})^{\gamma} + \lambda \right)^{\tau} - \lambda
}{\tau} \right]} u^{a-1}(1-u)^{b-1}du,
\end{eqnarray}
and
\begin{eqnarray}
\label{dnf7} S(x;a,b,\gamma,\beta,\lambda,\tau) = 1-
\frac{1}{B(a,b)}\int_{0}^{1- \exp\left[-\dfrac{\lambda^{1-\tau}
\left( (\frac{x}{\beta})^{\gamma} + \lambda \right)^{\tau} - \lambda
}{\tau} \right]} u^{a-1}(1-u)^{b-1}du
\end{eqnarray}
with $a>0$ and $b>0$, $0<x<\infty , 0<\gamma,\beta,\lambda<\infty ,
-\infty<\tau<\infty.$
\end{theorem}

\begin{proof}
From \eqref{dnf00}, we have for all $a>0$ and $b>0$
\begin{eqnarray*}
 F(G(x), a, b)
&=& I_{G(x)}(a,b)\\
&=& \frac{1}{B(a,b)}\int_{0}^{G(x)} u^{a-1}(1-u)^{b-1}du,
\end{eqnarray*}
where $I_{y}(a,b)= \frac{B_{y}(a,b)}{B(a,b)}$ is the incomplete beta ratio function and  $B_{y}(a,b)=\int_{0}^{y}
u^{a-1}(1-u)^{b-1}du$ the incomplete beta function.\\
Considering $G$ to be the cumulative distribution function of the four-parameter Weibull distribution, we obtain:
\begin{eqnarray*}
F(G(x), a, b)
&=& \frac{1}{B(a,b)}\int_{0}^{G(x)} u^{a-1}(1-u)^{b-1}du\\
&=& \frac{1}{B(a,b)}\int_{0}^{1-
\exp\left[-\dfrac{\lambda^{1-\tau} \left( (\frac{x}{\beta})^{\gamma}
+ \lambda \right)^{\tau} - \lambda }{\tau} \right]} u^{a-1}(1-u)^{b-1}du.
\end{eqnarray*}

From the expression \eqref{dnf6}, we define the associated survival function
which is given by $S(G(x), a, b)=1-F(G(x), a, b)$.\\
Therefore
\begin{eqnarray*}
S(G(x), a, b)
&=& 1-F(G(x), a, b)\\
&=& 1-\frac{1}{B(a,b)}\int_{0}^{1-
\exp\left[-\dfrac{\lambda^{1-\tau} \left( (\frac{x}{\beta})^{\gamma}
+ \lambda \right)^{\tau} - \lambda }{\tau} \right]} u^{a-1}(1-u)^{b-1}du.
\end{eqnarray*}
\end{proof}

The following result gives the 6-parameter probability density
of the new distribution:

\begin{theorem}
\label{theo1} 
The 6-parameter probability density function associated with the Beta modied Weibull distribution is given by :

\begin{eqnarray}
\label{dnf9}
 f(x;a,b,\gamma,\beta,\lambda,\tau)&=&\frac{\gamma}{\beta B(a,b)} \dfrac{(\frac{x}{\beta})^{\gamma-1}
 \lambda^{1-\tau}}
{\left[(\frac{x}{\beta})^{\gamma} + \lambda \right]^{1-\tau}}
\exp\left[-\dfrac{b\lambda^{1-\tau} \left(
(\frac{x}{\beta})^{\gamma}
+ \lambda \right)^{\tau} - \lambda }{\tau} \right]\nonumber\\
&\times& \left(1- \exp\left[-\dfrac{\lambda^{1-\tau} \left(
(\frac{x}{\beta})^{\gamma} + \lambda \right)^{\tau} - \lambda
}{\tau} \right] \right)^{a-1} \nonumber\\
\end{eqnarray}

\end{theorem}

\begin{proof}

By differentiating the expression \eqref{dnf00} with respect to $x$, we obtain :
En effet, on a :
\begin{eqnarray}
\label{dnf8}
f(G(x),a,b)&=& F'(G(x),a,b) \nonumber\\
&=& \left( \frac{1}{B(a,b)}\int_{0}^{G(x)} u^{a-1}(1-u)^{b-1}du\right)'  \nonumber\\
&=& \frac{1}{B(a,b)} G(x)^{a-1}[1-G(x)]^{b-1}g(x)
\end{eqnarray}
where $G(\cdot)$ and $g(\cdot)$ are given by the are given respectively by the
expressions \eqref{dnf3} and \eqref{dnf4}. \\
Thus,
\begin{eqnarray*}
f(x) &=&\frac{1}{B(a,b)} G(x)^{a-1}[1-G(x)]^{b-1}g(x)\\
&=&\frac{\gamma}{\beta B(a,b)} \dfrac{(\frac{x}{\beta})^{\gamma-1}
 \lambda^{1-\tau}}
{\left[(\frac{x}{\beta})^{\gamma} + \lambda \right]^{1-\tau}}
\exp\left[-\dfrac{\lambda^{1-\tau} \left(
(\frac{x}{\beta})^{\gamma}
+ \lambda \right)^{\tau} - \lambda }{\tau} \right]\\
&\times& \left(1- \exp\left[-\dfrac{\lambda^{1-\tau} \left(
(\frac{x}{\beta})^{\gamma} + \lambda \right)^{\tau} - \lambda
}{\tau} \right] \right)^{a-1}
\left(\exp\left[-\dfrac{\lambda^{1-\tau} \left(
(\frac{x}{\beta})^{\gamma} + \lambda \right)^{\tau} - \lambda
}{\tau} \right] \right)^{b-1}\\
&=&\frac{\gamma}{\beta B(a,b)} \dfrac{(\frac{x}{\beta})^{\gamma-1}
 \lambda^{1-\tau}}
{\left[(\frac{x}{\beta})^{\gamma} + \lambda \right]^{1-\tau}}
\exp\left[-\dfrac{\lambda^{1-\tau} \left(
(\frac{x}{\beta})^{\gamma} + \lambda \right)^{\tau} - \lambda
}{\tau} \right]^{b}\\
&\times& \left(1- \exp\left[-\dfrac{\lambda^{1-\tau} \left(
(\frac{x}{\beta})^{\gamma} + \lambda \right)^{\tau} - \lambda
}{\tau} \right] \right)^{a-1}
\\
&=&\frac{\gamma}{\beta B(a,b)} \dfrac{(\frac{x}{\beta})^{\gamma-1}
 \lambda^{1-\tau}}
{\left[(\frac{x}{\beta})^{\gamma} + \lambda \right]^{1-\tau}}
\exp\left[-\dfrac{b\lambda^{1-\tau} \left(
(\frac{x}{\beta})^{\gamma}
+ \lambda \right)^{\tau} - \lambda }{\tau} \right]\\
&\times& \left(1- \exp\left[-\dfrac{\lambda^{1-\tau} \left(
(\frac{x}{\beta})^{\gamma} + \lambda \right)^{\tau} - \lambda
}{\tau} \right] \right)^{a-1}\\
\end{eqnarray*}
\end{proof}

From the above, the following result is obtained, which gives the associated risk function:

\begin{corollary}
The hazard function associated with the 6-parameter Beta modified Weibull distribution is given by : 
{\footnotesize
\begin{eqnarray}
\label{dnf10}
 h(x;a,b,\gamma,\beta,\lambda,\tau)&=&
\dfrac{\gamma(\frac{x}{\beta})^{\gamma-1}
 \lambda^{1-\tau}}
{\beta
B(a,b)\left[1-I_{G(x;\gamma,\beta,\lambda,\tau)}(a,b)\right]\left[(\frac{x}{\beta})^{\gamma}
+ \lambda \right]^{1-\tau}} \exp\left[-\dfrac{b\lambda^{1-\tau}
\left( (\frac{x}{\beta})^{\gamma}
+ \lambda \right)^{\tau} - \lambda }{\tau} \right]\nonumber\\
&\times& \left(1- \exp\left[-\dfrac{\lambda^{1-\tau} \left(
(\frac{x}{\beta})^{\gamma} + \lambda \right)^{\tau} - \lambda
}{\tau} \right] \right)^{a-1} \nonumber\\
\end{eqnarray}
}
\end{corollary}

\begin{proof}
The proof of this corollary is direct. Indeed, since the
risk function is defined by $h(x)=\frac{f(x)}{S(x)}$, we have:
{\scriptsize
\begin{eqnarray*}
h(x;a,b,\gamma,\beta,\lambda,\tau)
&=& \dfrac{\frac{\gamma}{\beta B(a,b)} \dfrac{(\frac{x}{\beta})^{\gamma-1}
 \lambda^{1-\tau}}
{\left[(\frac{x}{\beta})^{\gamma} + \lambda \right]^{1-\tau}}
\exp\left[-\dfrac{b\lambda^{1-\tau} \left(
(\frac{x}{\beta})^{\gamma}
+ \lambda \right)^{\tau} - \lambda }{\tau} \right]
 \left(1- \exp\left[-\dfrac{\lambda^{1-\tau} \left(
(\frac{x}{\beta})^{\gamma} + \lambda \right)^{\tau} - \lambda
}{\tau} \right] \right)^{a-1}} 
{ 1-
\frac{1}{B(a,b)}\int_{0}^{1- \exp\left[-\dfrac{\lambda^{1-\tau}
\left( (\frac{x}{\beta})^{\gamma} + \lambda \right)^{\tau} - \lambda
}{\tau} \right]} u^{a-1}(1-u)^{b-1}du}\\ 
&=& \dfrac{\frac{\gamma}{\beta B(a,b)} \dfrac{(\frac{x}{\beta})^{\gamma-1}
 \lambda^{1-\tau}}
{\left[(\frac{x}{\beta})^{\gamma} + \lambda \right]^{1-\tau}}
\exp\left[-\dfrac{b\lambda^{1-\tau} \left(
(\frac{x}{\beta})^{\gamma}
+ \lambda \right)^{\tau} - \lambda }{\tau} \right]
 \left(1- \exp\left[-\dfrac{\lambda^{1-\tau} \left(
(\frac{x}{\beta})^{\gamma} + \lambda \right)^{\tau} - \lambda
}{\tau} \right] \right)^{a-1}} 
{ 1-
\frac{1}{B(a,b)}\int_{0}^{G(x;\gamma,\beta,\lambda,\tau)} u^{a-1}(1-u)^{b-1}du}\\ 
&=& \dfrac{\frac{\gamma}{\beta B(a,b)} \dfrac{(\frac{x}{\beta})^{\gamma-1}
 \lambda^{1-\tau}}
{\left[(\frac{x}{\beta})^{\gamma} + \lambda \right]^{1-\tau}}
\exp\left[-\dfrac{b\lambda^{1-\tau} \left(
(\frac{x}{\beta})^{\gamma}
+ \lambda \right)^{\tau} - \lambda }{\tau} \right]
 \left(1- \exp\left[-\dfrac{\lambda^{1-\tau} \left(
(\frac{x}{\beta})^{\gamma} + \lambda \right)^{\tau} - \lambda
}{\tau} \right] \right)^{a-1}} 
{ 1-I_{G(x;\gamma,\beta,\lambda,\tau)}(a,b)}\\
&=&
\dfrac{\gamma(\frac{x}{\beta})^{\gamma-1}
 \lambda^{1-\tau}}
{\beta
B(a,b)\left[1-I_{G(x;\gamma,\beta,\lambda,\tau)}(a,b)\right]\left[(\frac{x}{\beta})^{\gamma}
+ \lambda \right]^{1-\tau}} \exp\left[-\dfrac{b\lambda^{1-\tau}
\left( (\frac{x}{\beta})^{\gamma}
+ \lambda \right)^{\tau} - \lambda }{\tau} \right]\nonumber\\
&\times& \left(1- \exp\left[-\dfrac{\lambda^{1-\tau} \left(
(\frac{x}{\beta})^{\gamma} + \lambda \right)^{\tau} - \lambda
}{\tau} \right] \right)^{a-1} \\
\end{eqnarray*}
}

\end{proof}

\section{Simulation}
In this section, we plot the probability density and hazard functions according to certain values assigned to the parameters.

\subsection{Density function}

The following figures shows the graphical representation of the probability distribution function of the new distribution as well as those of its subdistributions :

\begin{figure}[H]
\centering
\begin{subfigure}[b]{0.4\linewidth}
\includegraphics[width=\linewidth]{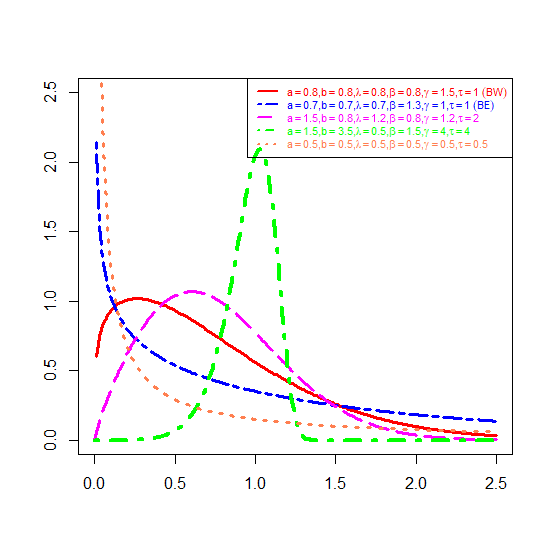}
\caption{} \label{Fig1a}
\end{subfigure}
\begin{subfigure}[b]{0.4\linewidth}
\includegraphics[width=\linewidth]{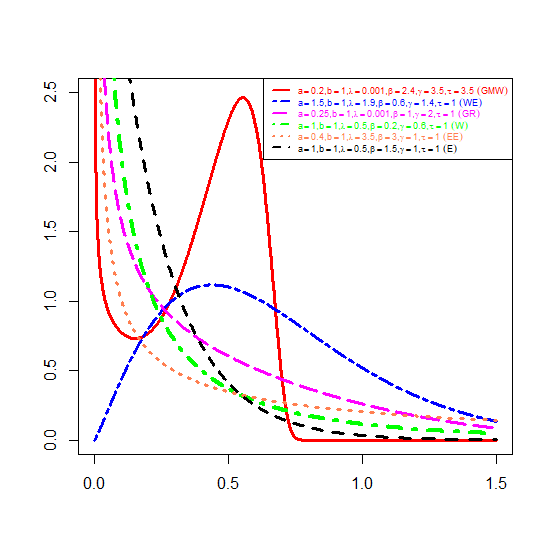}
\caption{} 
\label{Fig1b}
\end{subfigure}
\caption{Plots of the modified Beta weibull density for some parameter values.}
\label{figure1}
\end{figure}

\textbf{Interpretation} 
In figure \ref{Fig1a}, when $a=0.8, b=0.8, \lambda=0.8, \beta=0.8, \gamma=1.5, \tau=1$ we obtain the graphical representation (in red) of the density of the Beta Weibull (BW) distribution (see \cite{Famoye}).\\ 
When $a=0.7, b=0.7, \lambda=0.7, \beta=1.3, \gamma=1, \tau==1$ we obtain the
graphical representation (in blue) of the density of the exponential Beta distribution (BE) (see \cite{Nadarajah3}).\\
When $a=1.5, b=0.8, \lambda=1.2, \beta=0.8, \gamma=1.2, \tau=2$ (in purple), $a=1.5, b=3.5, \lambda=0.5, \beta=1.5, \gamma=4, \tau=4$ (in green), $a=0.5, b=0.5, \lambda=0.5, \beta=0.5, \gamma=0.5, \tau=0.5$ (in orange), we obtain the graphical representation of the density of the new six-parameter modified beta Weibull distribution.\\
In figure \ref{Fig1b}, when $a=0.2, b=1, \lambda=0.001, \beta=2.4, \gamma=3.5, \tau=3.5$ we obtain the graphical representation (in red) of the density of the Generalised Modified Weibull (GMW) distribution (see \cite{Carrasco}).\\
When $a=1.5, b=1, \lambda=1.9, \beta=0.6, \gamma=1.4, \tau==1$ we obtain the
graphical representation (in blue) of the density of the distribution
of the Weibull exponential (WE) (see \cite{Mudholkar1, Mudholkar2}).\\
When $a=0.25, b=1, \lambda=0.001, \beta=1, \gamma=2, \tau=1$ we obtain the
graphical representation (in purple) of the density of the generalized
Rayleigh (GR) distribution.\\
When $a=1, b=1, \lambda=0.5, \beta=0.2, \gamma=0.6, \tau==1$ we obtain the
graphical representation (in green) of the density of the Weibull
distribution (W) (see \cite{Lai}).\\
When $a=0.4, b=1, \lambda=3.5, \beta=3, \gamma=1, \tau==1$ We obtain the
graphical representation (in orange) of the density of the distribution
of the exponential function (EE) (see \cite{Gupta2}).\\
When $a=1, b=1, \lambda=0.5, \beta=1.5, \gamma=1, \tau=1$ we obtain the
graphical representation (in orange) of the density of the exponential
distribution (E) (see \cite{Zelen}).\\
We can therefore conclude from this figure that the new modified
six-parameter Beta Weibull distribution covers all density forms of
distributions in the literature.

\subsection{Hazard function}

The hazard function of the new 6-parameter distribution is represented
by the figures above :

\begin{figure}[H]
\centering
\begin{subfigure}[b]{0.4\linewidth}
\includegraphics[width=\linewidth]{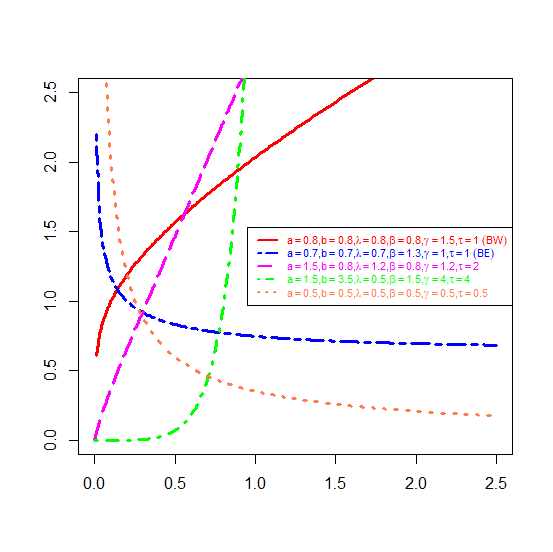}
\caption{} \label{Fig2a}
\end{subfigure}
\begin{subfigure}[b]{0.4\linewidth}
\includegraphics[width=\linewidth]{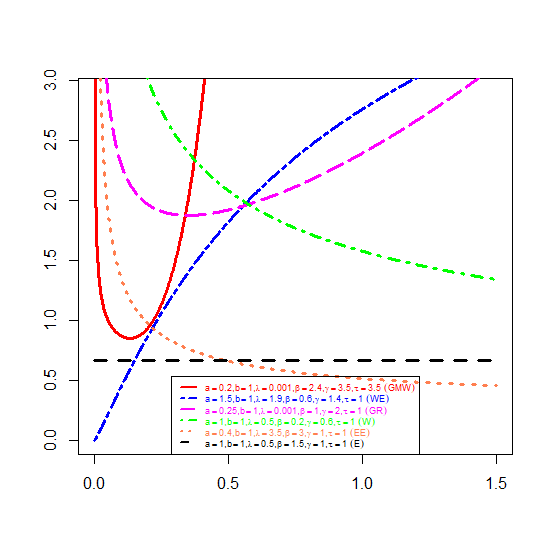}
\caption{} 
\label{Fig2b}
\end{subfigure}
\caption{Plots of the Modified Beta Weibull hazard
function for some parameter values} \label{figure2}
\end{figure}

\textbf{Interpretation} 
In figure \ref{Fig1a}, when $a=0.8, b=0.8, \lambda=0.8, \beta=0.8, \gamma=1.5, \tau=1$ we obtain the graphical representation (in red) of the hazard function of the Beta Weibull (BW) distribution (see \cite{Famoye}).\\ 
When $a=0.7, b=0.7, \lambda=0.7, \beta=1.3, \gamma=1, \tau==1$ we obtain the
graphical representation (in blue) of the hazard function of the exponential Beta distribution (BE) (see \cite{Nadarajah3}).\\
When $a=1.5, b=0.8, \lambda=1.2, \beta=0.8, \gamma=1.2, \tau=2$ (in purple), $a=1.5, b=3.5, \lambda=0.5, \beta=1.5, \gamma=4, \tau=4$ (in green), $a=0.5, b=0.5, \lambda=0.5, \beta=0.5, \gamma=0.5, \tau=0.5$ (in orange), we obtain the graphical representation of the hazard function of the new six-parameter modified beta Weibull distribution.\\
In figure \ref{Fig1b}, when $a=0.2, b=1, \lambda=0.001, \beta=2.4, \gamma=3.5, \tau=3.5$ we obtain the graphical representation (in red) of the hazard function of the Generalised Modified Weibull (GMW) distribution (see \cite{Carrasco}).\\
When $a=1.5, b=1, \lambda=1.9, \beta=0.6, \gamma=1.4, \tau==1$ we obtain the
graphical representation (in blue) of the hazard function of the distribution
of the Weibull exponential (WE) (see \cite{Mudholkar1, Mudholkar2}).\\
When $a=0.25, b=1, \lambda=0.001, \beta=1, \gamma=2, \tau=1$ we obtain the
graphical representation (in purple) of the hazard function of the generalized
Rayleigh (GR) distribution.\\
When $a=1, b=1, \lambda=0.5, \beta=0.2, \gamma=0.6, \tau==1$ we obtain the
graphical representation (in green) of the hazard function of the Weibull
distribution (W) (see \cite{Lai}).\\
When $a=0.4, b=1, \lambda=3.5, \beta=3, \gamma=1, \tau==1$ We obtain the
graphical representation (in orange) of the hazard function of the distribution
of the exponential function (EE) (see \cite{Gupta2}).\\
When $a=1, b=1, \lambda=0.5, \beta=1.5, \gamma=1, \tau=1$ we obtain the
graphical representation (in orange) of the hazard function of the exponential
distribution (E) (see \cite{Zelen}).\\

Figure \ref{figure2} above shows that all forms of hazard functions (constant,
increasing, decreasing, $\bigcup$-shaped and $\bigcap$-shaped) are represented
in the new distribution, which leads to the conclusion that
this new model studied fits almost all of the previous data and
encompasses all existing distributions.

\section{Special cases}

Here we give some specific cases of the new distribution, setting the
values of certain parameters. Indeed, by substituting $\tau=1,$ $\gamma=1$, $a=1$ et $b=1$ into the new associated functions, we observe the following :\\

When $\tau=1$, we find the Beta Weibull distribution, which we denote by BW (\cite{Wahed}).\\ 
When $ \tau=1, \gamma=1$, we find the exponential
beta distribution, which we denote by BE. (\cite{Nadarajah3}).\\ 
When $\tau=1, b=1$,we find the Weibull exponent, which we denote by EW (\cite{Mudholkar1}).\\ 
When $\tau=1, b=1, \gamma=1$, we find the exponent of the exponential function
which we denote by EE (\cite{Mudholkar1,Mudholkar}).\\ 
When $a=1, b=1$, we find the Weibull distribution with four parameters which we denote W4 (\cite{Jeong}).\\ 
When $a=1, b=1, \tau=1$, We find the Weibull distribution with two parameters which we denote W (\cite{Weibull}). 

The following table summarizes the sub-distributions obtained from the new distribution.

\begin{table}[H]
\centering
\begin{tabular}{|c|c|}
\hline \rule[-2ex]{0pt}{5.5ex} Distributions & Parameters \\
\hline \rule[-2ex]{0pt}{5.5ex} Six parameters Beta modified Weibull distribution & $a, b, \gamma, \beta, \lambda, \tau$  \\
\hline \rule[-2ex]{0pt}{5.5ex} Five parameters Beta modified Weibull distribution & $a, b, \gamma, \beta, \lambda, \tau\rightarrow \infty$  \\
\hline \rule[-2ex]{0pt}{5.5ex} Beta Weibull distribution & $a, b, \gamma, \beta, \lambda=1, \tau=1$  \\
\hline \rule[-2ex]{0pt}{5.5ex} Beta modified Rayleigh distribution & $a, b=1, \gamma=2, \beta, \lambda=1, \tau=1$ \\
\hline \rule[-2ex]{0pt}{5.5ex} Beta Rayleigh distribution & $a, b=1, \gamma=2, \beta=1, \lambda=1, \tau=1$ \\
\hline \rule[-2ex]{0pt}{5.5ex} Beta Exponential distribution & $a, b, \gamma=1, \beta, \lambda=1, \tau=1$ \\
\hline \rule[-2ex]{0pt}{5.5ex} Exponential Weibull distribution & $a, b=1, \gamma, \beta, \lambda=1, \tau=1$ \\
\hline \rule[-2ex]{0pt}{5.5ex} Exponent of the exponential distribution & $a, b=1, \gamma=1, \beta, \lambda=1, \tau=1$ \\
\hline \rule[-2ex]{0pt}{5.5ex} Four parameters Weibull distribution & $a=1, b=1, \gamma, \beta, \lambda, \tau$ \\
\hline \rule[-2ex]{0pt}{5.5ex} Three parameters Weibull distribution & $a=1, b=1, \gamma, \beta, \lambda, \tau\rightarrow 0$  \\
\hline \rule[-2ex]{0pt}{5.5ex} Modified Rayleigh distribution & $a=1, b=1, \gamma=2, \beta, \lambda=1, \tau=1$ \\
\hline \rule[-2ex]{0pt}{5.5ex} Classical Weibull distribution & $a=1, b=1, \gamma, \beta, \lambda=1, \tau=1$ \\
\hline \rule[-2ex]{0pt}{5.5ex} Rayleigh distribution & $a=1, b=1, \gamma=2, \beta=1, \lambda=1, \tau=1$ \\
\hline \rule[-2ex]{0pt}{5.5ex} Exponential distribution & $a=1, b=1, \gamma=1, \beta, \lambda=1, \tau=1$ \\
\hline
\end{tabular}
\caption{The new Beta modified Weibull distribution and its sub-distributions}
\label{tab1}
\end{table}
From the table above, we notice that by assigning the values
to the parameters of the six-parameter Beta modified Weibull distribution,
we obtain all the sub-distributions from the literature,
which allows us to conclude that the six-parameter modified Beta
Weibull distribution generalizes those from the literature.

\section{Conclusion}
\label{concl}

In this work, we have introduced a new six-parameter Beta modified
Weibull distribution generalizing the four-parameter Beta Weibull
distribution. Next, the distribution, survival, density, and hazard
functions of the new distribution were obtained. Finally, numerical
simulations of the probability density and hazard functions
were performed, highlighting two major aspects: The shapes of
the density function of the new modified six-parameter Beta modified
Weibull distribution cover all probability density forms of distributions
in the literature, and the graph of the hazard function
illustrates that all hazard function shapes (constant, increasing, decreasing,
in the form $\bigcup$ in the form $\bigcap$) are represented, meaning
that it encompasses all existing distributions. And for specific
cases, it appears that the new six-parameter Beta modified Weibull
distribution generalizes all sub-distributions in the literature.

Future work will focus on the mathematical properties of the
new six-parameter Beta modified Weibull distribution, such as linear
representation, order statistics, and moments. Parameter estimation
using the maximum likelihood method will also be studied.
The determination of asymptotic properties such as bias, variance,
the central limit theorem, and convergences (uniform and in probabilities)
will be addressed. We also plan to apply this to real-world
data to compare the efficiency and robustness of the new cumulative
distribution function.


\textbf{Author contributions }\\
All authors listed have made a substantial, and intellectual effort for this research.

\textbf{Acknowledgments  } \\

\textbf{Conflict of interest  }
The authors declare that the research was conducted in the absence of any commercial or financial relationships that could be construed as a potential conflict of interest.


\begin{thebibliography}{9}
\bibitem{Carrasco} Carrasco JMF, Ortega EMM and Cordeiro GM. (2008). A generalized modified
Weibull distribution for lifetime modeling. Computational Statistics and Data Analysis. 
\bibitem{Djongreba} Djongreba NF. (2025) Estimation problems based on dependent data. Theses, Université
de Maroua. https://hal.science/tel-05427772
\bibitem{Eugene} Eugene N, Lee C and Famoye F. (2002). Beta-normal Distribution and its applications. Com-
munication in Statistics - Theory and Methods, 31, 497-512.
\bibitem{Famoye} Famoye F, Lee C, Olumolade O. (2005) The beta Weibull distribution. Journal of
Statistical Theory and Applications; 4(2):121--36;
\bibitem{Jeong} Jeong JH (2006), A new parametric family
for modelling cumulative incidence function: application to breast
cancer data. J.R. Statistic. Soc. A, 169:289-303;
\bibitem{Gupta1} Gupta RD and Kundu D. (1999). Generalized exponential distributions. Austral. New Zeal.
J. Statist., 41, 173-188.
\bibitem{Gupta2} Gupta RD and Kundu D. (2001). Exponentiated exponential distribution: an alternative to
gamma and Weibull distributions. Biometrical Jour., 43, 117-130.
\bibitem{Khan} Khan MN (2015). The modified beta Weibull distribution, Hacettepe Journal of Mathematics and Statistics. 44(6):1553-1568.
\bibitem{Kundu} Kundu D and Rakab MZ. (2005). Generalized Rayleigh distribution: different methods of
estimation. Computational Statistics and Data Analysis, 49, 187-200.
\bibitem{Lai} Lai CD, Xie M and Murthy DNP. (2003). A modified Weibull distribution. Transactions
on Reliability, 52, 33-37.
\bibitem{Lee} Lee C, Famoye F and Olumolade O (2007) Beta-Weibull Distribution: Some Properties and Applications to Censored Data, Journal of Modern Applied Statistical Methods: Vol. 6 : Iss. 1 , Article 17.
\bibitem{Mudholkar1} Mudholkar GS and Srivastava DK. (1993). Exponentiated Weibull family for analyzing
bathtub failure-real data. IEEE Transaction on Reliability, 42,
299-302.
\bibitem{Mudholkar2} Mudholkar GS, Srivastava DK and Friemer M. (1995). The exponentiated Weibull family:
A reanalysis of the bus-motor-failure data. Technometrics, 37, 436-445.
\bibitem{Mudholkar} Mudholkar GS, Srivastava DK and Kollia GD (1996). A generalization of the Weibull distribution with application to the analysis of survival data. J. Amer. Statist. Assoc., 91, 1575-1583;
\bibitem{Nadarajah2} Nadarajah S and Kotz S. (2004), The beta Gumbel distribution, Mathematical Problems
in Engineering, 1, 323-332.
\bibitem{Nadarajah3} Nadarajah S and Kotz S. (2006). The beta exponential distribution. Reliability Engineering
and System Safety, 91, 689-697.
\bibitem{Nadarajah1} Nadarajah S, Teimouri M and Shih SH.(2014) Modified Beta Distributions. Sankhya Ser. B 76,
19-48.
\bibitem{Ngatchou} Ngatchou-Wandji, J., Ltaifa, M., Njamen Njomen, D. A., Shen, J. (2022). Nonparametric Estimation of the Density Function of the Distribution of the Noise in CHARN Models. Mathematics, 10(4), 624.
\bibitem{Njomen} Njomen DAN, Yayebga HC (2019). Density and Risk Function
of the Circular Kernel Study, European Journal of Pure and
Applied Mathematics, 12(4), 1612-1642.
\bibitem{Silva} Silva GO, Edwin MM, Ortega and Cordeiro GM (2010). The
beta modified Weibull distribution, Lifetime Data Anal. 16, 409-430;
\bibitem{Wahed} Wahed AS, Luong TM, and Jeong J-H. (2009) A new
generalization of Weibull distribution with application to a breast
cancer data set. Statistics in Medicine; 28(16): 2077-2094;
\bibitem{Weibull} Weibull WA. (1951). Statistical distribution function of wide applicability. ASME
Journal of Applied Mechanics Transactions of the American Society of
Mechanical Engineers:293--7.
\bibitem{Zelen} Zelen M and Dannemiller MC. (1961). The robustness of life testing procedures derived from the exponential distribution. Technometrics 3, 29-49. 


\end{thebibliography}
    \end{document}